\documentclass[12pt]{amsart}
\usepackage{graphicx, overpic}
\usepackage[below]{placeins}

\usepackage{amsmath,amscd,amssymb, epsfig,psfrag}

\newtheorem{theorem}{Theorem}[section]
\newtheorem{lemma}[theorem]{Lemma}

\newtheorem{cor}[theorem]{Corollary}
\newtheorem{Def}[theorem]{Definition}

\newtheorem{conjecture}[theorem]{Conjecture}

\newcommand{\ZZ}{\mathbb{Z}}
\newcommand{\NN}{\mathbb{N}}
\newcommand{\QQ}{\mathbb{Q}}

\newcommand{\RR}{\mathbb{R}}

\newcommand{\bbslash}{\backslash\backslash}

\def\b{\beta}

\def\g{\gamma}

\def\G{\Gamma}

\def\S{\Sigma}

\def\PSL{\mbox{\rm{PSL}}}

\def\SL{\mbox{\rm{SL}}}

\def\min{\mbox{\text{min}}}

\def\HH{\mathbb{H}}

\def\interior{{\rm int}}

\def\min{\mbox{\rm{min}}}

\def\Int{\mbox{int}\ }

\def\HH{\mathbb{H}}

\def\split{\backslash\backslash}

\edef\t@mp{\catcode`\noexpand\#=\the\catcode`\#}%
    \catcode`\#=12
    \def\h@sh{#}%
    \t@mp

\edef\t@mp{\catcode`\noexpand\~=\the\catcode`\~}%
    \catcode`\~=12
    \def\tild@{~}%
    \t@mp

\begin{document}

\title{Criteria for virtual fibering} 

\author[Ian Agol]{%
        Ian Agol} 
\address{%
      University of California, Berkeley \\ 970 Evans Hall \# 3840 \\ Berkeley, CA 94720-3840}
\email{%
        ianagol@gmail.com}  

\thanks{Agol partially supported by NSF grant DMS-0504975, and the J. S. Guggenheim Foundation}



\begin{abstract} 
We prove that an irreducible 3-manifold whose fundamental
group satisfies a certain group-theoretic property called {\it RFRS} is virtually fibered. 
As a corollary, we show that 3-dimensional reflection orbifolds
and arithmetic hyperbolic orbifolds defined by a quadratic form
virtually fiber. These include the Seifert Weber dodecahedral
space and the Bianchi groups. Moreover, we show that a taut sutured compression
body has a finite-sheeted cover with a depth one taut-oriented
foliation.
\end{abstract} 

\maketitle
\section{Introduction}
Thurston proposed the question of whether a hyperbolic 3-manifold 
has a finite-sheeted cover fibering over $S^1$ \cite{Th:82}. 
A connected fibered manifold with fiber of Euler characteristic $<0$ must be irreducible. However, there exist non-fibered
closed irreducible 3-manifolds which have no cover which fibers over $S^1$,
the simplest example being a Seifert fibered space with non-zero Euler
class and base orbifold of non-zero Euler characteristic \cite{OrlikRaymond69}. 
Since Thurston posed this question, many 
classes of hyperbolic 3-manifolds have been proven to virtually fiber. 
Thurston proved that the
reflection group in a right-angled hyperbolic dodecahedron virtually
fibers. 

Gabai gave an example of a union of two $I$-bundles over non-orientable
surfaces, which has a 2-fold cover which fibers, but in some sense this 
example is too simple since it fibers over a mirrored interval (this example was known to Thurston). He also gave an example of a 2 component link which is not fibered
and has a 2-fold cover which fibers, which is not of the previous type of example \cite{Gabai86}. Aitchison and Rubinstein
generalized Thurston's example to certain cubulated 3-manifolds \cite{AR:99}. Reid gave
the first example of a rational homology sphere which virtually fibers using arithmetic methods \cite{Reid95}.
Leininger showed that every Dehn filling on one component of the Whitehead
link virtually fibers \cite{Leininger02}. Walsh showed that all of the non-trivial 2-bridge link complements
virtually fiber \cite{Walsh05}. There are more recent unpublished results of DeBlois and
Agol, Boyer, and Zhang exhibiting classes of Montesinos links which virtually fiber by generalizing Walsh's method. Most of these examples are shown to be virtually fibered by exhibiting a fairly explicit cover of the manifold
which fibers. Button found many virtually fibered manifolds
in the Snappea census of cusped and closed hyperbolic 3-manifolds \cite{Button05}. Of course, we have neglected mention of the
many papers exhibiting classes of 3-manifolds which fiber. 

A characterization of virtual fibering was given by Lackenby \cite{Lackenby04}. He showed
that if $M$ is a closed hyperbolic manifold, with a tower of regular covers such that the
Heegaard genus grows slower than the fourth root of the index of the cover, then
$M$ virtually fibers. It is difficult, however, to apply this criterion to a specific
example in order to prove it virtually fibers since it is difficult to estimate the Heegaard genus of covering spaces. 

In this paper, we give a new criterion for virtual fibering which is a condition on 
the fundamental group called RFRS. Using this condition, and recent results of 
Haglund-Wise \cite{HaglundWise07}, we prove that hyperbolic reflection groups
and arithmetic groups defined by quadratic forms virtually fiber. 
Rather than state the RFRS condition in the introduction, we
state the following theorem, which is a corollary of Theorem \ref{RFRSvfiber} and
Corollary \ref{RFRS}: 
\begin{theorem}
Let $M$ be a compact irreducible orientable 3-manifold with
$\chi(M)=0$. If $\pi_1(M)$ is a subgroup of a right-angled Coxeter
group or right-angled Artin group, then $M$ virtually fibers. 
\end{theorem}

The strategy of the proof is intuitive. If an oriented surface $F\subset M$ is non-separating, then 
$M$ fibers over $S^1$ with fiber $F$ if and only if $M\split F \cong I\times F$. 
If $F$ is not a fiber of a fibration, then there is a JSJ decomposition of $M\split F$ \cite{JS76,Johannson79},
which has a product part (window) and a non-product part (guts). The idea is to produce
a complexity of the guts, and use the RFRS condition to produce a cover
of $M$ to which a component of the guts lifts and for which we can decrease the complexity of the guts, by ``killing'' it using non-separating
surfaces coming from new homology in this cover. The natural setting for these ideas
is sutured manifold hierarchies \cite{Ga2}, but we choose instead to use least-weight taut normal 
surfaces introduced by Oertel \cite{Oertel86} and further developed by Tollefson and Wang \cite{TollefsonWang96}. The replacement
for sutured manifold hierarchies is branched surfaces, whose machinery was introduced by
Floyd and Oertel \cite{FO}. We review some of the results from these papers, but in order
to follow the proofs in this paper, one should be familiar with the results of \cite{Oertel86,TollefsonWang96}.

{\bf Acknowledgments:} We thank Alan Reid for encouraging us to try to prove that the
Bianchi groups virtually fiber, for sharing the preprint \cite{LongReid07}, and for letting us know of Theorem \ref{FD}.
We also thank Misha Kapovich for helpful comments and we thank Haglund and Wise for sharing their preprint
\cite{HaglundWise07}. We further thank the referee for helpful
comments, and Michelle McGuinness for proofreading.

\section{Residually finite rational solvable groups}

The rational derived series of a group $G$ is defined
inductively as follows. If $G^{(1)} = [G,G]$, then
$G^{(1)}_r = \{ x\in G |\exists k \neq 0, x^k\in G^{(1)} \}$.
If $G^{(n)}_r$ has been defined, define
$G^{(n+1)}_r =( G^{(n)}_r)^{(1)}_r$. The rational derived
series gets its name because $G^{(1)}_r = \ker \{ G \to \QQ\otimes_{\ZZ} G/G^{(1)} \}$. The quotients $G/G^{(n)}_r$ are solvable. 

\begin{Def}
A group $G$ is {\it residually finite $\QQ$-solvable} or {\it RFRS} 
if there
is a sequence of subgroups $G=G_0 > G_1 > G_2 > \cdots$ 
such that $G\rhd G_i$,  $\cap_i G_i =\{1\}$, $[G:G_{i}] < \infty$ and $G_{i+1}\geq (G_i)^{(1)}_r$.  
\end{Def} 
By induction, $G_i \geq G^{(i)}_r$, and thus $G/G_i$ is solvable with derived series of length at most
$i$. 
We remark that if $G$ is RFRS, then any subgroup
$H<G$ is as well.
Also, we remark that to check that $G$ is RFRS, we need only find a cofinal sequence of finite index subgroups $G> G_i$ such that 
$G_{i+1} > (G_i)^{(1)}_r$. This follows because $G\rhd Core(G_i )=\cap_{g\in G} g G_i g^{-1}$. 
Then the sequence $G\rhd Core(G_1)\rhd Core(G_2) \rhd \cdots$ will satisfy
the criterion by the following argument. 
For $g\in G$, $g ((G_i)^{(1)}_r) g^{-1}=(g G_i g^{-1})^{(1)}_r$,
so
$$Core(G_{i+1}) = \cap_{g\in G} g G_{i+1} g^{-1}\geq \cap_{g\in  G} (g G_i g^{-1})^{(1)}_r \geq (Core(G_i))^{(1)}_r.$$ 
Thus, the new sequence satisfies the properties of the RFRS criterion. 
The relevance of this definition may not be immediately apparent, but it
was discovered by the following method. Suppose we have a 
manifold (or more generally a topological space) $M$ with
$\pi_1 M=G$ such that
$rk H_1(M;\QQ)> 0$. Choose a non-trivial finite quotient 
$H_1(M;\ZZ)/Torsion \to H_0$, and take the cover $M_1$ of $M$ 
corresponding to $ker\{\pi_1(M) \to H_0\}$. Suppose that $rk H_1(M_1;\QQ)>rk H_1(M;\QQ)$.
Then we may take a finite index abelian cover $M_2 \to M_1$ coming from
a quotient of $H_1(M_1;\ZZ)/Torsion$ which does not factor through
$H_1(M_1)\to H_1(M)$. If we can repeat this process to get a 
sequence of covers $M\leftarrow M_1 \leftarrow M_2 \leftarrow \cdots$
such that $M_{i+1}$ is a finite abelian cover of $M_i$, and the sequence
of covers is cofinal (so that the limiting covering space is the universal
cover), then $\pi_1(M)$ is RFRS. Notice
that we may assume that $rk H_1(M_i;\QQ)$ is unbounded,
unless $\pi_1(M)$ is virtually abelian. Thus, we see
that a manifold $M$ such that $\pi_1(M)$ is not virtually abelian but
is RFRS has
virtual infinite first betti number.
This formulation of the RFRS condition is the
characterization that we will apply to study 3-manifolds. 

\begin{theorem} \label{coxeterRFRS}
A finitely generated right-angled Coxeter group $G$ has a finite index
subgroup $G'$ such that $G'$ is RFRS. 
\end{theorem}
\begin{proof}
Let $G$ act on a convex subset $C\subset\RR^n$ with quotient a polyhedral orbifold $D=C/G$, 
such that each involution generator of $G$ fixes a codimension one hyperplane in $C$. 
This is possible by taking $C$ to be homeomorphic to the Tits cone for the
group $G$ \cite[V.4]{Bourbaki68}, \cite[5.13]{Humphreys90}. 
We may take a subgroup $G'<G$ of finite index inducing a manifold cover $D'\to D$, such that for each codimension one
face $F$ of $D$, the preimage of $F$ in $D'$ is embedded, orientable and
2-sided. For example, we may take $G'= \ker\{ G \to G/G^{(1)}\}$,
since $G/G^{(1)} \cong (\ZZ/2\ZZ)^{rk G}$. 
Consider a cofinal sequence of finite-index covers $D\leftarrow D_1 \leftarrow D_2 \leftarrow \cdots$
such that $D_{i+1}$ is a 2-fold cover of $D_i$ obtained by reflecting in a face
$F_i\subset D_i$ (see Figure \ref{coxeter}). Let $G_i=\pi_1(D_i)\cap G'$, with quotient manifold $D_i'=C/G_i$. 
We need to show that $(G_i)^{(1)}_r < G_{i+1}$. Clearly, $(G_i)^{(1)}< G_{i+1}$, since
$G_i/G_{i+1}\cong \ZZ/2\ZZ$ or $G_{i+1}=G_i$. Thus, we need only show that
any element mapping non-trivially to $G_i/G_{i+1}$ is not torsion, so that
we have a factorization $G_i \to H_1(G_i)/Torsion \to \ZZ/2\ZZ$. 
Suppose $g\in G_i - G_{i+1}$. Take a loop $\gamma$ representing $g$ in $D'_i$,
and project $\gamma$ down to $D_i$ generically. We see a closed path bouncing 
off the faces of $D_i$ (see Figure \ref{coxeter}). 
Then the projection of  $\gamma$ must hit $F_i$ an odd number of times, otherwise it would lift to $D_{i+1}$
and therefore would be contained in $G_{i+1}$. 
We therefore see $\gamma$ intersecting the preimage of $F_i$ in $D'_i$ an odd number of times.
Since the preimage of $F_i$ is an orientable embedded 2-sided codimension 1 submanifold, we see that $\gamma$
must represent an element of infinite order in $H_1(G_i)=G_i/(G_i)^{(1)}$, and 
therefore $0\neq 2[\gamma] \in H_1(D_i';\QQ)=\QQ\otimes G_i/(G_i)^{(1)}_r$.
This implies that $g^2\notin (G_i)^{(1)}_r$. But $g^2\in G_{i+1}$, since $[G_i:G_{i+1}]\leq 2$.
Therefore, we have $G_i/G_{i+1}\cong \ZZ/2\ZZ$ is a quotient of $H_1(G_i;\ZZ)/Torsion$.
Thus, we see that $G'$ is RFRS. 
\end{proof}

\begin{figure}[htb] 
	\begin{center}
	\psfrag{a}{$D$}
	\psfrag{b}{$D_1$}
	\psfrag{c}{$D_2$}
	\psfrag{d}{$\gamma$}
	\psfrag{e}{$\gamma^2$}
	\psfrag{f}{$D_3$}
	\psfrag{g}{$\cdots$}
	\psfrag{h}{$F_1$}
	\psfrag{i}{$F_0$}
	\epsfig{file=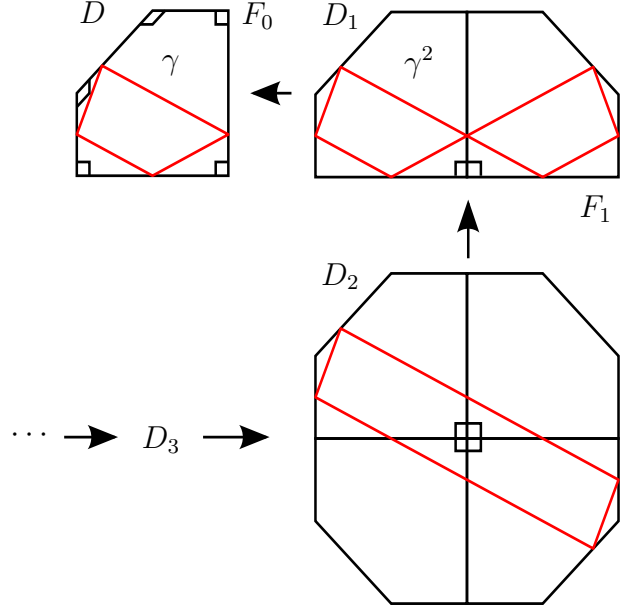, width=.5\textwidth, height=.5\textwidth}
	\caption{\label{coxeter} A right-angled polyhedral orbifold $D$ and a tower of 2-covers}
	\end{center}
\end{figure}

\begin{cor} \label{RFRS}
The following groups are virtually RFRS:
\begin{itemize}
\item surface  groups
\item reflection groups
\item right-angled Artin groups
\item arithmetic hyperbolic groups defined by a quadratic form 
\item direct products and free products of RFRS groups 
\end{itemize}
\end{cor}
\begin{proof}
Surface groups are subgroups of the reflection group 
in a right-angled pentagon (see e.g. \cite{S}). Reflection groups, right-angled Artin groups, and
arithmetic hyperbolic groups defined by a quadratic form all have finite index subgroups which are
subgroups of right-angled reflection groups by  \cite[p. 3]{HaglundWise07} (this was
proven for cusped arithmetic 3-manifolds previously in \cite{ALR}), and are thus virtually RFRS by Theorem \ref{coxeterRFRS}.

Direct products of RFRS groups are 
immediately seen to be RFRS. 
If $G, H$ are RFRS, then we have $G_i\lhd G$, $[G:G_i]<\infty$, such that
$(G_i)^{(1)}_r < G_{i+1}$, and $H_i\lhd H$, $[H:H_i]<\infty$, such that $(H_i)^{(1)}_r<H_{i+1}$.
Then $G_i\times H_i \lhd G\times H$ is finite index, and $(G_i\times H_i)^{(1)}_r =(G_i)^{(1)}_r\times (H_i)^{(1)}_r \lhd G_{i+1}\times H_{i+1}$.

Free products of RFRS groups are also
easily seen to be RFRS. Let $G$ and $H$
be (non-trivial) RFRS as above, with $G=G_0\rhd G_1\rhd G_2\rhd  \cdots$,
$H=H_0\rhd H_1\rhd H_2\rhd \cdots$. We construct 
inductively groups $GH_i \lhd G\ast H$ such that $GH_i \cong (H_i)^{\ast k(i)} \ast (G_i)^{\ast j(i)}\ast \ZZ^{\ast l(i)}$.
To obtain $GH_{i+1}$ from $GH_i$, take the kernel of the quotient of $GH_i$
onto $ (H_i/H_{i+1})^{k(i)}\times (G_i/G_{i+1})^{j(i)}\times (\ZZ/2\ZZ)^{l(i)}$. This may 
be easily seen by induction of the action on the Bass-Serre tree. The free product
$G\ast H$ is a graph of groups with graph a single edge with trivial stabilizer,
and two vertices with stabilizers $G$ and $H$, respectively. The Bass-Serre
tree $T$ is infinite bipartite. When we take the quotient $T/GH_i$, we get a graph
of groups with bipartite graph, and vertices corresponding to conjugates of $G_i$ and $H_i$,
which is how we see the claimed Kurosh decomposition, where the rank of
the free factor $l(i)$ corresponds to the first betti number of the graph $T/GH_i$. For any element 
$f\in G\ast H$, we may write $f$ (up to conjugacy) as a cyclically reduced 
product $f=  g_1h_1 g_2 h_2 \cdots g_m h_m$, where $g_j\neq 1, h_j\neq 1$, 
unless $f$ is conjugate into $G$ or $H$, in which case it is easily seen that $f\notin GH_i$ for
some $i$.  
Choose $i$ such that $g_j \notin G_i, h_j \notin H_i$, $j=1, \ldots, m$. 
Then the projection of a loop representing $f$ to the graph $T/GH_i$
will be homotopically non-trivial, since each time it enters a vertex
of $T/GH_i$, it will exit a distinct edge, since $g_j$ and $h_j$ are
not in the vertex stabilizers of $T/GH_i$. Thus, the projection of 
$f$ to $\ZZ^{\ast l(i)}$ will be non-trivial, and since the free group is 
RFRS (with respect to 2-group quotients),
we see that $f$ will eventually not be in $GH_i$. 
\end{proof}

\section{Sutured manifolds}
We consider 3-manifolds that are triangulated and orientable, and PL
curves and surfaces.
For a proper embedding $Y\subset X$, define $X\bbslash Y$
to be the manifold $\overline{X\backslash \mathcal{N}(Y)}$,
where $\mathcal{N}(Y)$ denotes a regular neighborhood of $Y$. 

\begin{Def}
Let $S$ be a compact connected orientable surface. Define
$\chi_{-}(S)=\max\{-\chi(S), 0\}$. For a disconnected compact orientable surface
$S$, let $\chi_-(S)$ be the sum of $\chi_-$ evaluated on the connected subsurfaces of $S$.
\end{Def}
In \cite{Th3}, Thurston defines a pseudonorm on $H_2(M,\partial M;\RR)$
and $H_2(M;\RR)$, and this was generalized by Gabai \cite{Ga2}.

\begin{Def}
Let $M$ be a compact oriented 3-manifold. Let $K$ be a subsurface
of $\partial M$. Let $z\in H_2(M,K)$. Define the {\rm norm of z}
to be
$$x(z)=\min\{\chi_-(S)\ |\ (S,\partial S)\subset (M,K),\ \text{and} \ [S]=z\in H_2(M,K)\}.$$
\end{Def}
Thurston proved that $x(nz)=n x(z)$, and therefore $x$ may
be extended to a norm $x:H_2(M,K;\QQ)\to \QQ$. Then $x$
is extended to $x:H_2(M,K;\RR)\to \RR$ by continuity.

We have the Poincar\'e duality isomorphisms 
$PD: H^1(M)\to H_2(M,\partial M)$ and $PD:H_2(M,\partial M)\to H^1(M)$, so
that $PD^2=Id$. 
We may define for $\alpha \in H^1(M)$, $x(\alpha)= x(PD(\alpha))$.
We say that $\alpha$ is a fibered cohomology class if $c\cdot PD(\alpha)$
is represented by a fiber for some constant $c$. Denote by $\mathcal{B}_x(M)$ the unit ball of the
Thurston norm on $H^1(M)$ or on $H_2(M,\partial M)$. 

\begin{Def}
Let $S$ be a properly embedded oriented surface in the compact
oriented 3-manifold $M$. Then $S$ is {\it taut} in
$H_2(M,K)$ if $\partial S\subset K$, $S$ is incompressible,
there is no homologically trivial union of components of $S$, and
$\chi_-(S) =x([S])\ \text{for}\ [S]\in H_2(M,K)$.
\end{Def}

\begin{Def}
A {\it taut (codimension 1) foliation} is a foliation which has a closed transversal
going through every leaf.
\end{Def}

The following theorem says that a compact leaf of a taut oriented
foliation is a taut surface. This was proven by Thurston,
under the assumption that the foliation is differentiable \cite{Th3}.
Gabai generalized this to deal with general foliations \cite{Ga4}.

\begin{theorem}
Let $M$ be a compact oriented 3-manifold. Let $F$ be a 
taut oriented foliation  of $M$ such
that $F$ is transverse to $\partial M$. If $R$ is a compact leaf of $F$, then
$R$ is taut.
\end{theorem}

If $M$ is a 3-manifold fibering over $S^1$ with fiber $F$,
then there is a cone over a face $E$ of $\mathcal{B}_x(M)$ whose interior contains $[F]$ 
and such that every other rational homology class in $\interior E$ is also
a fibered class \cite{Th3}.

Gabai introduced the notion of a sutured manifold in order to prove
the converse of the above theorem. 
\begin{Def}
A sutured manifold $(M,\gamma)$ is a compact oriented 3-manifold
$M$ together with a set $\gamma\subset \partial M$ of pairwise
disjoint annuli $A(\gamma)$ and tori $T(\gamma)$. Furthermore,
the interior of each component of  $A(\gamma)$ contains a
{\it suture}, {\it i.e.}, a homologically
nontrivial oriented simple closed curve. We denote the set
of sutures by $s(\gamma)$. Finally, every component of $R(\gamma)=
\partial M \bbslash A(\gamma)$ is oriented. Define $R_+(\gamma)$
(or $R_-(\gamma)$) to be those components of $R(\gamma)$ whose
normal vectors point out of (into) $M$. The orientations on $R(\gamma)$
must be coherent with respect to $s(\gamma)$, {\it i.e.}, if
$\delta$ is a component of $\partial R(\gamma)$, and is given the boundary
orientation, then $\delta$ must represent the same homology
class in $H_1(\gamma)$ as the suture in the component of $A(\gamma)$ containing $\delta$.
\end{Def}

\begin{theorem} \label{vd1fiber}
Let $M$ be a 3-manifold, and let $f\in H_2(M,\partial M)$ be a homology 
class which lies in the boundary of the cone over a fibered face $E$ of $\mathcal{B}_x(M)$.
Let $(F,\partial F)\subset (M,\partial M)$ be a taut surface representing
$f$. Then $M\split F$ is a taut sutured manifold which admits a depth
one taut oriented foliation. 
\end{theorem}

\begin{proof}
Let $G$ be a surface such that $[G]\in \interior E$ is a surface fibering
over $S^1$. Since $x([G]+[F])=x([G])+x([F])$, we may make $G$ and
$F$ transverse so that $\chi_-(G\oplus F)=\chi_-(G)+\chi_-(F)$, where $G\oplus F$ indicates
oriented cut and paste. Then $F\cap (M\split G)$ gives a taut sutured
manifold decomposition of $M\split G$ \cite{Ga2}. Since $G$ is a fibered surface,
$M\split G$ is a product. Any taut sutured decomposition of a product
sutured manifold is also a product sutured manifold, so $M\split (G\cup F)$
is a product sutured manifold, with sutures corresponding to $G\cap F$ \cite{Ga2}. 
Consider a product structure on $\mathcal{N}(F)\cong [-1,1]\times F$,
such that $G\cap \mathcal{N}(F)\cong [-1,1]\times (G\cap F)$. Take
infinitely many copies of $F$, so that we have $\{\pm 1/n, n\in \NN\} \times F\subset \mathcal{N}(F)$. 
Then we may ``spin'' $G$ about $F$ by adding $G-(G\cap F)$ to infinitely
many copies of $F$. We take $G-(G\cap F) \oplus ( \{\pm 1/n, n\in \NN\} \times F)$
to get an infinite leaf $L$ which spirals about $\{0\} \times F \subset \mathcal{N}(F)$ (see Figure \ref{spinning}). 
Then $(M\split F) \split L \cong I\times L$, since it is obtained by adjoining
to the product sutured manifold $M\split (G\cup F)$ infinitely many 
copies of a product sutured manifold of the form $I \times (F\split (G\cap F))  $
along annuli. Thus, $M\split F$ has a depth one taut oriented foliation. 
\end{proof}

 \begin{figure}[htb] 
	\begin{center}
	\psfrag{a}{$\{1\}\times F$}
	\psfrag{b}{$\{1/2\} \times F$}
	\psfrag{c}{$\{1/3\} \times F$}
        \psfrag{d}{$\{-1/3\} \times F$}
        	\psfrag{e}{$\{-1/2\} \times F$}
	\psfrag{f}{$\{-1\} \times F$}
	\psfrag{g}{$\vdots $}
	\psfrag{h}{$\vdots$}
	\psfrag{G}{$G$}
	\psfrag{F}{$F$}
	\psfrag{N}{$\mathcal{N}(F)$}
	\psfrag{L}{$L$}
	\epsfig{file=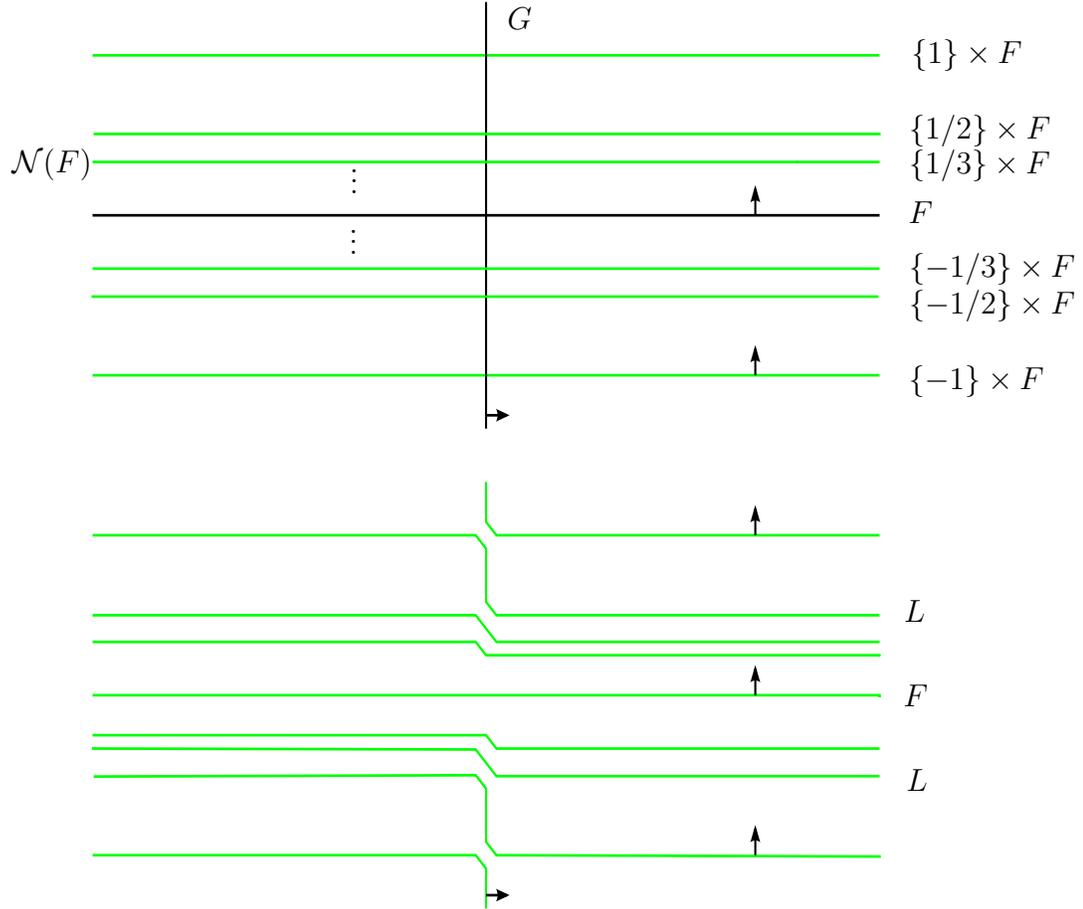, width=.75\textwidth, height=.75\textwidth}
	\caption{\label{spinning} Spinning $G$ about $F$ to obtain $L$ }
	\end{center}
\end{figure}

\section{Normal surfaces}

We'll follow the notation of \cite{TollefsonWang96}. 
Let $M$ be a compact 3-manifold with a triangulation $\mathcal{T}$ with $t$ tetrahedra.
A normal surface $F$ in $M$ is a properly embedded surface in general position with the
1-skeleton $\mathcal{T}^{(1)}$ and intersecting each tetrahedron $\Delta$ of $\mathcal{T}$ in properly embedded elementary
disks which intersect each face in at most one straight line. A normal isotopy of
$M$ is an isotopy which leaves the simplices of $\mathcal{T}$ invariant. If $E$ is an elementary disk then $\partial E$
is uniquely determined by the points $E\cap \mathcal{T}^{(1)}$, and a collection of normal disks in
$\Delta$ is uniquely determined up to normal isotopy by the number of points of
$E\cap\Delta^{(1)}$ in each edge of $\Delta^{(1)}$.
A normal surface $F$ is uniquely determined up to normal isotopy by the set of points $F\cap \mathcal{T}^{(1)}$. An elementary disk is determined, up to normal isotopy, by the manner in which it separates the vertices of $\Delta$ and we
refer to the normal isotopy class of an elementary disk as its disk type. There are seven
possible disk types in each tetrahedron corresponding to the seven possible
separations of its four vertices; four of which consist of triangles and three consisting of
quadrilaterals. The normal isotopy class of an arc in which an elementary disk meets
a triangle of $\Delta$ is called an arc type.

We fix an ordering $(d_1 , . . . , d_{7t})$ of the disk types in $\mathcal{T}$ and assign
a 7$t$-tuple ${\bf x} = (x_1, . . . , x_{7t})$, called the normal coordinates of $F$, to a normal surface $F$ by
letting $x_i$ denote the number of elementary disks in $F$ of type $d_i$. The normal surface $F$ is
uniquely determined, up to normal isotopy, by its normal coordinates. Among 7$t$-tuples of
nonnegative integers ${\bf x} = (x_1,. . . , x_{7t})$, those corresponding to normal surfaces are characterized by two constraints. The first constraint is that it must be possible to realize the
required 4-sided disk types $d_i$ corresponding to nonzero $x_i$Õs by disjoint elementary disks.
This is equivalent to allowing no more than one quadrilateral disk type to be represented in each
tetrahedron. The second constraint concerns the matching of the edges of elementary disks
along incident triangles of tetrahedra. Consider two tetrahedra meeting along a common
2-simplex and fix an arc type in this 2-simplex. There are exactly two disk types from each of the
tetrahedra whose elementary disks meet this 2-simplex in arcs of the given arc type. If the
7$t$-tuple is to correspond to a normal surface then there must be the same number of
elementary disks on both sides of the incident 2-simplex meeting it in arcs of the given type. This
constraint can be given as a system of matching equations, one equation for each arc type in
the 2-simplices of $\mathcal{T}$ interior to $M$.

The matching equations are given by
$$x_i + x_j = x_k + x_l,\  0\leq x_i,$$
where $x_i, x_j, x_k, x_l$ share the same arc type in a 2-simplex of $\mathcal{T}$.
The nonnegative solutions to the matching equations form an infinite linear cone
$\mathcal{S}_{\mathcal{T}} \subset \RR^{7t}$. The normalizing equation $\sum_{i=1}^{7t} x_i = 1$ is added to form the system of normal
equations for $\mathcal{T}$. The solution space 
$\mathcal{P}_{\mathcal{T}}\subset \mathcal{S}_{\mathcal{T}}$ becomes a compact, convex, linear cell and is
referred to as the projective solution space for $\mathcal{T}$. The projective class of a normal surface $F$ is
the image of the normal coordinates of $F$ under the projection $\mathcal{S}_{\mathcal{T}}\to\mathcal{P}_{\mathcal{T}}$. A rational point
${\bf z}\in \mathcal{P}_{\mathcal{T}}$ is said to be an admissible solution if corresponding to each tetrahedron there is at
most one of the quadrilateral variables which is nonzero. Every admissible solution is the
projective class of an embedded normal surface.

Given a normal surface $F$, we denote by $C_F$ the unique minimal face of $\mathcal{P}_{\mathcal{T}}$ that contains
the projective class of $F$. A normal surface is carried by $C_F$ if its projective class belongs to
$C_F$. It is easy to see that a normal surface $G$ is carried by $C_F$ if and only if there exists
a normal surface $H$ such that $mF = G + H$ for some positive integer $m$, where $+$ indicates addition
of normal surfaces, which is the unique way to perform cut and paste on $G\cap H$ to obtain another
normal surface when $G$ and $H$ have compatible quadrilateral types. Every rational point
in $C_F$ is an admissible solution, since it will have the same collection of quadrilateral types
as $F$. 

The weight of a normal surface $F$ is the number of points in
$F\cap \mathcal{T}^{(1)}$, which will be denoted by $wt(F)$ 
(this was introduced in \cite{Dunwoody85}, \cite{Oertel86}, and \cite{JacoRubinstein88} as a combinatorial
substitute for the notion of the area of a surface).
The type of a normal surface $F$ is the collection of types
of non-zero quadrilaterals and triangles in the surface $F$.   
Associated to a normal surface $F$ is a branched surface $\overline{B}_F$, obtained by identifying 
normally parallel disks in each tetrahedron, and neighborhoods of arcs in each
2-simplex, and neighborhoods of vertices in each edge, and then perturbing to make generic (see Figure \ref{normalbranched}). 
For a branched surface $B$ we will define $Guts(B)=M\split B$. If $F$
is a normal surface with associated branched surface $B$, then $Guts(B)$
consists of the non-normally product parts of $M\split F$, and gives a 
combinatorial substitute for $Guts(M\split F)$, which consists of the
non-product pieces of the JSJ decomposition.

A taut surface $F$ is said to be {\it lw-taut} (standing for {\it least weight-taut})
if it has minimal
weight among all taut surfaces representing the homology class
$[F]$. If $F$ is lw-taut, then $n$ pairwise disjoint copies of 
$F$, denoted by $nF$, is an lw-taut surface representing the
class $n[F]$. 
Let $\mathcal{B}_x(M) \subset H_2(M,\partial M;\RR)$ be the
unit ball of the Thurston norm on homology. Tollefson and 
Wang prove that associated to each non-trivial homology
class $f\in H_2(M,\partial M;\ZZ)$, there is a face $C_f$
of $\mathcal{P}_{\mathcal{T}}$, which is called a 
{\it complete lw-taut} face, which has the following properties
(\cite[Theorem 3.7]{TollefsonWang96}). Let $[C_f]$ denote
the set of homology classes of all oriented normal surfaces
carried by $C_f$. Then $C_f$ carries every lw-taut normal
surface representing any homology class in $[C_f]$. 
We extend the definition so that for $\alpha\in H^1(M)$, 
$C_{\alpha}=C_{PD(\alpha)}$. 

The construction of $C_f$ is given in the following way.
For a normal surface $F$, the unique minimal face $C_F$
of $\mathcal{P}_{\mathcal{T}}$ carrying the projective class
of $F$ is said to be {lw-taut} if every surface carried by 
$C_F$ is lw-taut. Tollefson and Wang prove that if 
$F$ is lw-taut, then $C_F$ is lw-taut \cite[Theorem 3.3]{TollefsonWang96}. 
Then for $f\in H_2(M,\partial M;\ZZ)$, one takes the finite set
$\{F_1, \ldots, F_n\}$ up to normal isotopy of all oriented,
lw-taut normal surfaces representing $f$, and let 
$F=F_1+\cdots +F_n$ be the lw-taut surface representing
the homology class $nf$, which is canonically associated
to $f$ up to normal isotopy. The addition here is both
addition of normal surfaces and oriented cut-and-paste
addition, which preserves the homology class. 
Then define $C_f\equiv C_F$, 
which will have the desired properties. One may also
associate the branched surface $\overline{B}_f\equiv \overline{B}_F$
canonically to the homology class $f$. 

If $F$ is a lw-taut surface, then there is canonically associated a normal branched
surface $B_F$, which is obtained from $\overline{B}_F$
by splitting along punctured disks of contact. Tollefson and
Wang prove that $B_F$ is a Reebless incompressible branched
surface which is oriented. If $F$ is the canonical
lw-taut normal surface associated to $f$, then $B_f\equiv B_F$
has the property that every surface carried by $C_F$ 
is isotopic to a surface carried by $B_f$, such that the
orientation on the surface is induced by the orientation of $B_f$.
This follows from the proof of \cite[Lemma 6.1]{TollefsonWang96},
where they prove that every surface carried by $C_F$ is
homologous to a surface carried by $B_f$, by cut-and-paste
of normal disks. However, in the case that $M$ is irreducible,
cut-and-paste of disks may be achieved by an isotopy. 
The branched surface $\overline{B}_F$ associated to $F$ has a fibered neighborhood $\overline{N}_F=\overline{N}_f$,
and $B_F$ has a fibered neighborhood $N_F=N_f$. 
Moreover, we may isotope $B_F$ to be normally carried
by $\overline{N}_F$ so that it is transverse to the fibers. 
In this case $Guts(B_f)=M\split N_f$ is a taut-sutured manifold,
with $\gamma(B_f)=\partial_v(N_F)$, and $R(\gamma)=\partial_h(N_F)$,
such that the orientation of each component of $R(\gamma)$ is induced
by the orientation of $B_f$. From the construction, suppose that
$g\in [C_f]$, then $C_g\subseteq C_f$, $[C_g]\subseteq [C_f]$,
and $\overline{B}_g \subseteq \overline{B}_f$ (up to normal
isotopy). 
For each component $Q$ of $Guts(\overline{B}_f)$, there is a complexity
$c(Q)$ which is the number of normal disk types $D$ 
such that $D$ is compatible with the disk types of $C_f$,
and $D$ cannot be normally isotoped into $\overline{N}_f$ transverse
to the fibers to miss
$Q$ (see Figure \ref{normalbranched}). 
Define 
$$c(Guts(\overline{B}_f) ) = \max \{ c(Q) | Q\  {\mbox{\rm is a component of }} Guts(\overline{B}_f)\}.$$
 
 \begin{figure}[htb] 
	\begin{center}
	\psfrag{0}{$0$}
	\psfrag{1}{$1$}
	\psfrag{2}{$4$}
	\epsfig{file=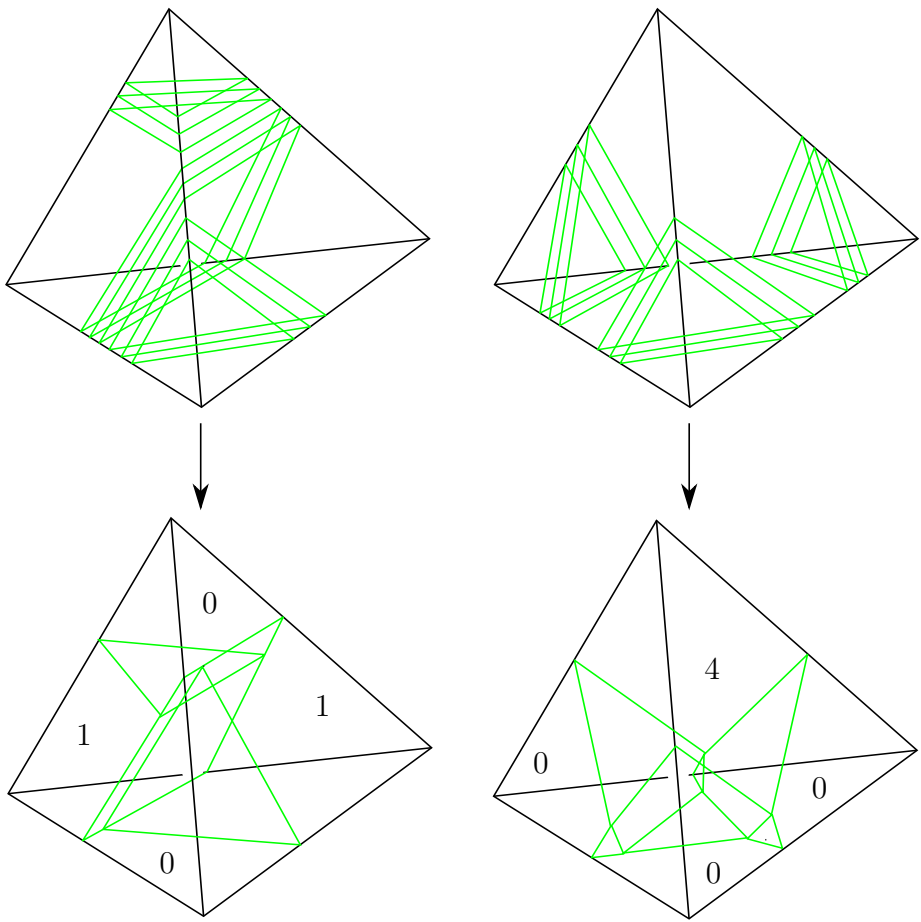, width=.75\textwidth, height=.75\textwidth}
	\caption{\label{normalbranched} Local contribution to complexity of the guts of a branched surface obtained by pinching normal disks}
	\end{center}
\end{figure}

\begin{lemma} \label{killing}
Let $M$ be a compact irreducible orientable 3-manifold. 
Suppose that $f \in H_2(M,\partial M)$ is a homology
class lying in the interior of a cone $E$ over a face of $\mathcal{B}_x(M)$. Moreover
assume that $F \in \interior(C_f)$, where $F$ is the canonical
projective representative of the homology class $f$ introduced above, and
$C_f$ is a maximal lw-taut face. Let $B_f$ be
the taut oriented normal branched surface associated 
to the cone $C_f$, and let $Guts(B_f)$ be the
complementary regions of $N_f$. 
Then $im\{ H^1(M)\to H^1(Guts(B_f))\} = 0$.  
\end{lemma} 
\begin{proof}
Suppose that $im\{ H^1(M)\to H^1(Guts(B_f))\} \neq 0$. 
Let $g \in H_2(M,\partial M;\ZZ)$ be a homology class
dual to a cohomology class $PD(g)\in H^1(M)$ with
non-zero image in $H^1(Guts(B_f))$. Then for 
large enough $k$, we have $k f+g \in E$, since $f$ lies
in the interior of $E$, and we may assume that 
$kf+g \in [C_f]$.  But then any lw-taut surface
$S$ such that  $[S]=kf+g$ is isotopic to one carried
by $B_f$ by Lemma 6.1 \cite{TollefsonWang96}. 
This gives a contradiction, since for any surface $S$ carried
by $B_f$, we must have for any oriented loop $\alpha \subset Guts(B_f)$,
$\alpha \cap S =\emptyset$. But there is an oriented loop $\alpha \subset Guts(B_f)$
such that $PD(g)(\alpha)\neq 0$, or equivalently $g\cap [\alpha]\neq 0$.
But since $f\cap [\alpha]=0$, this implies that $[S]\cap [\alpha] = (nf+g)\cap [\alpha] =g\cap [\alpha]\neq 0$,
a contradiction. 
\end{proof}

\begin{cor} \label{killing2}
Under the same hypotheses as Lemma \ref{killing},
$im\{ H^1(M)\to H^1(Guts(\overline{B}_f))\} =0$.
\end{cor}
\begin{proof}
The branched surface $B_f$ is obtained from $\overline{B}_f$
by splitting along punctured disks of contact. Thus, we
have an embedding $Guts(\overline{B}_f)\subset Guts(B_f)$,
which gives a factorization $H^1(M)\to H^1(Guts(B_f)) \to H^1(Guts(\overline{B}_f))$.
Then we apply Lemma \ref{killing} to see that the image is $0$. 
\end{proof}

\begin{lemma} \label{product}
Suppose that $f\in H_2(M,\partial M;\ZZ)$ is a homology 
class such that $(Guts(B_f),\gamma(B_f))$  is a product
sutured manifold. Then $f$ is a fibered homology class.
\end{lemma}
\begin{proof}
Let $F$ be an lw-taut surface fully carried by $B_f$ representing
$nf$ as above. Associated
to $B_f$ is an $I$-bundle $\overline{L}_F$ as described in Section
6 of \cite{TollefsonWang96}, which is obtained by splitting 
$\overline{N}_F$ along $2F-int_{2F}(\partial_h\overline{N}_F)$. 
If $(Guts(B_f),\gamma(B_f))$ is a product sutured manifold,
then the $I$-bundle structures on $\overline{L}_F\cup Guts(B_f)$
match together to give an $I$-bundle structure on the complement of $2F$,
which implies that $F$ is a fiber. 
\end{proof}

\section{Virtual fibering}
The following criterion for virtual fibering is the principal result in this paper. 

\begin{theorem} \label{RFRSvfiber}
Let $M$ be a connected orientable irreducible 3-manifold with $\chi(M)=0$ such that $\pi_1(M)$
is RFRS. If $0\neq f\in H^1(M)$ is a non-fibered homology class,
then there exists a finite-sheeted cover $p_{i,0}:M_i\to M$ 
such that  $p_{i,0}^{\ast} f \in H^1(M_i)$ lies in the cone over the boundary of a fibered
face of $\mathcal{B}_x(M_i)$.
\end{theorem}

\begin{proof} 
Let $G=\pi_1(M)$, with triangulation $\mathcal{T}$. We are assuming that  $G$ is RFRS, so there exists
$G=G_0 \rhd G_1 \rhd G_2 \rhd \cdots$ such that $G \rhd G_i$, $[G:G_i]<\infty$, and $G_{i+1}>(G_i)^{(1)}_r$ so that $G_i/G_{i+1}$
is abelian. Let $M_i$ be the regular cover of $M$ corresponding to the subgroup $G_i$, so that $M_0=M$. 
First, we give an outline of the argument. Choose a norm-minimizing
surface $\S_0$ in $M_0$ so that $[\S_0]$ lies in a maximal face of $\mathcal{B}_x(M_0)$, and whose projective
class is a small perturbation of $PD(f)$. Each component of the guts of $\S_0$ will lift to $M_1$ by Lemma \ref{killing}. We successively
lift each component of the guts of $\S_0$ to $M_i$ until a component
of the guts has a non-trivial image in $H^1(M_i)$. We
perturb the projective homology class of the preimage of $\S_0$
in $M_i$ again, to get a surface with simpler gut regions. 
The new gut regions will be obtained by sutured manifold
decomposition from the original gut regions of $\S_0$,
and in fact they embed back down into $M$.
The technical heart of the proof is the machinery to show that
these sutured manifold decompositions terminate. Gabai
showed that any particular taut sutured manifold hierarchy will terminate, but his complexity does not bound a priori the number of possible decompositions. Since
we are in some sense performing multiple sutured
manifold decompositions ``in parallel'' in each cover, 
it is not guaranteed that the family of decompositions
will terminate. 
Thus we use the machinery of normal surfaces. The
complexity measures roughly how many normal disks
types are still available in each gut region to decompose
along.

For $j>i$, let $p_{j,i}:M_j\to M_i$ be the covering map, and let $\mathcal{T}_i=p_{i,0}^{-1}(\mathcal{T})$ be the induced
triangulation of $M_i$. 
Let our given cohomology class $f\in H^1(M_0)$ be denoted by $f_0$. 
We construct by induction cohomology classes $f_i\in H^1(M_i)$, $i>0$
such that the canonical lw-taut surface $F_i$ such that $[F_i]=n_i PD(f_i)$ lies in the interior of $C_{f_i}$.
Given $f_i$, choose $f_{i+1}\in H^1(M_{i+1})$ to satisfy the
following properties:
\begin{itemize}
\item the projective class of $f_{i+1}$ is a small perturbation of $p_{i+1,i}^{\ast} f_i$ in the unit norm ball $\mathcal{B}_{x}(M_{i+1})$

\item
$C_{p_{i+1,i}^{\ast} f_i}\subset C_{f_{i+1}}$ is a facet

\item
$f_{i+1}$ lies in a maximal face of the unit norm 
ball $\mathcal{B}_x(M_{i+1})$ 

\item
there is a lw-taut surface representative $F_{i+1}$ of $n_{i+1} PD(f_{i+1})$ so that $F_{i+1}\in \interior( C_{f_{i+1}})$ and $C_{f_{i+1}}$ is a maximal face.
\end{itemize}
This concludes the inductive construction of the cohomology classes
$f_i$. We now prove that for $i$ large enough, $f_i$ is a fibered
cohomology class in $H^1(M_i)$.

To prove this, we show that if $f_i$ is not a fibered cohomology 
class, then there exists $J>i$ such that $c(Guts(\overline{B}_{f_J})) < c(Guts(\overline{B}_{f_i}))$.
By Lemma \ref{killing}  we have
$im\{H^1(M_i)\to H^1(Guts(B_{f_i}))\} =0$. 
Thus each component of $Guts(B_{f_i})$ 
lifts to $M_{i+1}$. 
Let $Q\subseteq Guts(B_{f_{i}})$ be a connected component such that 
$\pi_1(Q) \neq 1$. Recall that since $B_{f_i}$ is a Reebless incompressible
branched surface, $\pi_1 Q\hookrightarrow \pi_1 M_i$ is injective \cite{O2}. 
Let $\overline{Q}=Q\cap Guts(\overline{B}_{f_i})$,
where we are assuming that $B_{f_i} \subset \overline{N}_{f_i}$ is
carried transverse to the fibers. 
Since $\cap_j G_j=1$, there
exists $j=j(Q)> i$ such that $ G_{j+1}\cap \pi_1 (Q) \neq \pi_1(Q) \leq G_{j}$
(since $G_{j+1}\lhd G$, notice that we do not need to worry
about the conjugacy class of $\pi_1 Q < \pi_1 M= G$). Then there
exists $g\in \pi_1(\overline{Q})$ such that the projection of $g$ to the abelian group $G_{j}/G_{j+1}$
is non-trivial. Therefore $g$ represents a homologically non-trivial
element in $H_1(M_{j})$ of infinite order. If $F_i\subset M_i$ is an lw-taut
surface such that $[F_i]=n_i PD(f_i)$, then the preimage $p_{j,i}^{-1}(F_i) \subset M_{j}$ of $F_i$ representing the homology class $n _iPD(p_{j,i}^{\ast}f_i)\in H_2(M_{j},\partial M_{j})$ is also an lw-taut surface by \cite[Theorem 3.2]{To93}. We see that the preimage $p_{j,i}^{-1}(\overline{B}_{f_i}) = \overline{B}_{p_{j,i}^{\ast}f_i} \subset \overline{B}_{f_{j}} \subset M_{j}$. 

Thus, $Q$ will lift to $M_{j}$,
but $\overline{Q}$ will be decomposed by $\overline{B}_{f_j}$ into $\overline{Q}_j$ for which $c(\overline{Q}_j)< c(\overline{Q})$.
Otherwise $\overline{Q}$ is a component of $Guts(\overline{B}_{f_j})$, which would contradict Corollary \ref{killing2}, since $im\{ H^1(M_j)\to H^1(Q)\}\neq 0$. 
 Since $\overline{B}_{f_{j}}$
is obtained by identifying parallel normal disks, and $\overline{B}_{p_{j,i}^{\ast}f_{i}}\subset \overline{B}_{f_{j}}$,
there must be a normal disk type of $C_{f_{j}}$ which does not appear
as a normal disk type of $C_{p_{j,i}^{\ast}f_{i}}$, and thus ``kills'' the lift of $g$ from $M_i$ to $M_{j}$. In particular, this normal disk must lie inside of
the lift of $\overline{Q}$ to $M_{j}$, in such a way that it may not be normally isotoped to
lie inside of $\overline{N}_{p_{j,i}^{\ast}f_i}$. Since $M_{j}$ is a regular cover of $M_i$, this
implies that the same holds for every lift of $\overline{Q}$ to $M_{j}$. 
Now, we choose $J$ to be the maximum of $j(Q)$ over
all components $Q$ of $Guts(B_i)$. 
Thus, $c(Guts(\overline{B}_J))< c(Guts(\overline{B}_i))$. 

If $\pi_1(Q) =1$, then $Q$ must be a taut sutured ball. If every component
$Q$ of $Guts(B_{f_i})$ is a taut-sutured ball, then by Lemma \ref{product},
$f_i$ is a fibered cohomology class. By the induction above, we
see that there exists $I$ such that for all $i\geq I$, $c(Guts(\overline{B}_i))= c(Guts(\overline{B}_I))$. But this implies that for every
component $Q$ of $Guts(B_I)$, $\pi_1(Q)=1$, and therefore
$f_I$ is a fibered cohomology class. 
\end{proof}

The following theorem is a corollary of Theorem \ref{RFRSvfiber} and Corollary \ref{RFRS}. 
\begin{theorem} \label{vquadfiber}
If $\mathcal{O}$ is a 3-dimensional hyperbolic reflection orbifold of finite volume or
an arithmetic hyperbolic orbifold defined by a quadratic form, then there exists an orientable manifold
cover $\tilde{\mathcal{O}}\to\mathcal{O}$ such that $\tilde{\mathcal{O}}$ fibers over $S^1$. 
\end{theorem}

{\bf Remark:} Arithmetic groups defined by quadratic forms
include the Bianchi groups $\PSL(2, \mathcal{O}_d)$,
where $\mathcal{O}_d$ is the ring of integers in a quadratic
imaginary number field $\QQ(\sqrt{-d})$ where $d$ is a
positive integer, and the Seifert-Weber dodecahedral
space. These examples are some of the oldest known examples of hyperbolic lattices.

\section{Virtual depth one}
The next theorem gives an analogue of Theorem \ref{RFRSvfiber} for taut sutured manifolds. 
A taut sutured manifold $(M,\g)$ with $\chi(M)<0$ virtually fibers over an interval if and only if it is a product sutured manifold.
If $(M,\g)$ has a depth one taut-oriented foliation $\mathcal{F}$, then $\mathcal{F}_{|\interior M}$ gives
a fibering of $\interior M$ over $S^1$ if the corresponding cohomology class is rational. Thus,
depth one seems to be the best suited analogue to fibering for taut sutured manifolds with boundary.

\begin{theorem} \label{vd1}
Let $(M,\gamma)$ be a connected taut sutured manifold such that $\pi_1(M)$
satisfies the RFRS property. Then there is a finite-sheeted
cover $(\tilde{M},\tilde{\gamma}) \to (M, \gamma)$ such that $(\tilde{M},\tilde{\gamma})$
has a depth one taut oriented foliation. 
\end{theorem}
\begin{proof}
The proof is similar to the proof of Theorem \ref{RFRSvfiber}.
We double $M$ along $R(\g)$ to get a manifold $DM_{\g}$ with
(possibly empty) torus boundary, which admits an orientation
reversing involution $\tau$ exchanging the two copies of $M$,
and fixing $R_{\pm}(\g)$ while reversing their coorientations. In
order that we may apply the lw-taut technology, we choose a triangulation
$\mathcal{T}$ of $DM_{\g}$  for which the only lw-taut surfaces representing
$[R_{\pm}(\g)]$ are isotopic to $R_{\pm}(\g)$. This is possible
by choosing a triangulation of $DM_{\g}$ for which $R_{\pm}(\g)$
are normal of the same weight and such that $R_{+}(\g)\cup R_{-}(\g)$
intersects each tetrahedron in at most one normal disk, and each tetrahedron
meeting $R_{\pm}(\g)$ does not intersect any other tetrahedron meeting $R_{\pm}(\g)$
except in the faces containing arcs of a normal disk of $R_{\pm}(\g)$.
In this case, if another normal surface is homologous to $R_{\pm}(\g)$, then it must
meet a tetrahedron which does not meet $R_{+}(\g)\cup R_{-}(\g)$, since the tetrahedra
meeting $R_{\pm}(\g)$ form a product neighborhood. 
Now, 
stellar subdivide the 
tetrahedra not meeting $R_{\pm}(\g)$ a number of times, until
every normal surface $F$ homologous to $R_{\pm}(\g)$ has weight $wt(F)>wt(R_{\pm}(\g))$. 
This is the triangulation that we work with on $DM_{\g}$ 
in order that we may apply the lw-taut machinery without
modification. 

The retraction $r:DM_{\g}\to M$ induces a retraction $r_{\#}:\pi_1 DM_{\g}\to \pi_1 M$. 
Choose orientations on the surfaces $R_{\pm}(\g)$
so that they are homologous, and call the resulting
lw-taut surface $F$. Then $[F]=2[R_{+}]$. 
Also, $B_F = \overline{B}_F=F$, and $Guts(B_F)=DM_{\g}\split F$.
Notice that for each component $Q$ of $Guts(B_F)$,
$\pi_1(Q)< \pi_1(M)$, up to the action of $\tau$. 
We proceed now as in the proof of Theorem \ref{RFRSvfiber}.
There are covers $M_i\to M$ satisfying the RFRS criterion.
We construct covers $N_i\to DM_{\g}$ so that 
$\pi_1(N_i)=r_{\#}^{-1}( \pi_1(M_i))$, that is induced by
the retraction $r:DM_{\g}\to M$. 
We may apply 
the proof of Theorem \ref{RFRSvfiber} to 
find a cover $p:N_i\to DM_{\g}$ and a cohomology class $f\in H^1(N_i)$ which is a perturbation of $p^{\ast}(PD([F]))$ such that
$f$ is a fibered class, by induction on $c(Guts(\overline{B}_{f_i}))$. We also have that each component
of $N\split p^{-1}(F)$ is a cover of $M$. By Theorem \ref{vd1fiber}, $N\split p^{-1}(F)$ has a depth one
taut oriented foliation. Thus, $M$ has a finite sheeted cover with a 
depth one taut oriented foliation. 
\end{proof}

The following theorem is a corollary of Corollary \ref{RFRS} and Theorem \ref{vd1}.
\begin{theorem}
Let $(M,\gamma)$ be a taut sutured compression body. Then there is a finite-sheeted
cover $(\tilde{M},\tilde{\gamma}) \to (M, \gamma)$ such that $(\tilde{M},\tilde{\gamma})$
has a depth one taut oriented foliation. 
\end{theorem}

{\bf Remark:} The previous theorem is non-trivial, since there exist examples due to 
Brittenham of taut sutured genus 2 handlebodies which admit no depth one taut
oriented foliation \cite{Brittenham}.

\section{Fibered faces}
In this section, we prove a result of Long and
Reid \cite{LongReid07}, which says that a fibered arithmetic 3-manifold has finite index covers with arbitrarily many
fibered faces of the Thurston norm unit ball. 
The first result of this type was proved by Dunfield and Ramakrishnan for a specific arithmetic 3-manifold \cite{DR07}. 
It follows from Theorem \ref{vquadfiber} that this holds for any arithmetic 3-manifold defined by a
quadratic form (without the added assumption of fibering),
although this will also follow from Theorem \ref{RFRSfaces}. The proof is inspired by the strategy
of Long and Reid's argument, where we replace their use of pseudo-Anosov flows with harmonic
2-forms. 
\begin{theorem}
For any $k>0$, an arithmetic 3-manifold $M$ which fibers over $S^1$ has a finite-sheeted
cover $M_k\to M$ such that $\mathcal{B}_x(M_k)$ has at least $k$ distinct fibered faces. 
\end{theorem}
\begin{proof}
For simplicity, assume that $M$ is closed. Let $\alpha\in H_2(M)$ be a homology class which fibers. Then there is an Euler class
$e\in H^2(M)$ such that for any $\beta\in H_2(M)$ lying in the same face of $\mathcal{B}_x(M)$
as $\alpha$, one has $e(\beta)=x(\beta)$ \cite{Th3}. Consider a harmonic representative $e_0\in \mathcal{H}^2(M)$,
where $\mathcal{H}^2(M)\cong H^2(M;\RR)$ is the space of harmonic 2-forms on $M$ with respect to the hyperbolic metric. 
Let $\tilde{e}_0\in \mathcal{H}^2(\HH^3)$ be the image of $e_0$ in the universal cover $\HH^3\to \HH^3/\G= M$. By the
proof of \cite[Theorem 0.1]{Agol06}, the orbit of $\tilde{e}_0$ under $Comm(\G)$ is infinite.
We finish again as in the proof of \cite[Theorem 0.1]{Agol06}. For any $k$, choose $g_{1},\ldots, g_{k}\in Comm(\G)$
such that $\{ g_{j\ast}(\tilde{e}_0) | j=1,\ldots, k \}$ are distinct. Then  
let $\G_k= \cap_{j} (g_{j}\G g_{j}^{-1})$. The 2-forms $g_{j\ast} \tilde{e}_0$ project 
to distinct harmonic 2-forms on $M_k=\HH^3/\G_k$, each of which is dual to a distinct
fibered face of the Thurston norm unit ball which is the preimage of a fibered 
face on $\HH^3/g_j\G g_j^{-1}\cong M$ associated to the projection of $g_{j\ast} \tilde{e}_0$.
\end{proof}

Next, we prove that an irreducible 3-manifold $M$ with $\chi(M)=0$
and $\pi_1(M)$ RFRS has virtually infinitely many fibered faces. 

\begin{theorem} \label{RFRSfaces}
Let $M$ be a compact orientable irreducible 3-manifold with
$\chi(M)=0$ and $\pi_1(M)$ virtually RFRS and not virtually
abelian. Then for any $k>0$, there exists a cover $M_k\to M$
such that $\mathcal{B}_x(M_k)$ has at least $k$ distinct fibered faces. 
\end{theorem}
\begin{proof}
Since $\pi_1(M)$ is RFRS, we may assume that $M$ fibers by 
Theorem \ref{RFRSvfiber}. We will prove the result by induction. For $k>0$, assume $M$ has a
cover $M_k \to M$ such that $\mathcal{B}_x(M_k)$ has $\geq k$ fibered faces. Since $\pi_1(M)$
is RFRS and not virtually abelian, there exists a cover 
$M'_k\to M_k$ such that $\b_1(M'_k)>\b_1(M_k)$ and 
$\mathcal{B}_x(M'_k)\subset H^1(M'_k)$ has more faces
than $\mathcal{B}_x(M_k)\subset H^1(M_k)$. 
If $\mathcal{B}_x(M'_k)$ has all fibered faces, then $M'_k$
has $>k$ fibered faces since by hypothesis it has more
faces than $\mathcal{B}_x(M_k)$. Otherwise, $\mathcal{B}_x(M'_k)$
has a non-fibered face $U$, such that there is an Euler class
$e\in H^2(M'_k)$ so that $x(z)=e(PD(z))$ for $z\in U$. 
Fix $z\in \Int(U)$ and take $p:M''_k\to M'_k$ such
that $p^{\ast} z$ lies in the boundary of a fibered face $U'$
of $\mathcal{B}_x(M''_k)$ by Theorem \ref{RFRSvfiber}. Let $e'\in H^2(M''_k)$ be the
Euler class associated to $U'$. Then $e'(PD(p^{\ast} z))=p^{\ast}(e)(PD(p^{\ast} z))$. But then for any other fibered face $V$ 
of $\mathcal{B}_x(M'_k)$ with associated Euler class $f$,
we have $x(z)\neq f(PD(z))$. Thus, $x(PD(p^{\ast} z))\neq p^{\ast}(f)(PD(p^{\ast} z))$, and therefore $e' \neq p^{\ast} f$. So $U'$
must be a fibered face not containing $p^{\ast}(V)$, and
therefore $\mathcal{B}_x(M''_k)$ has at least $k+1$ fibered faces. 

Geometrically, what is happening here is that the polyhedron
$p^{\ast} \mathcal{B}_x(M'_k)$ is the intersection of the subspace
$p^{\ast} H^1(M'_k) \subset H^1(M''_k)$
with  $\mathcal{B}_x(M''_k)$. Each fibered face of $p^{\ast}\mathcal{B}_x(M'_k)$ will lie in the interior of a fibered face of $\mathcal{B}_x(M''_k)$, whereas the non-fibered face $ p^{\ast}U$ will
lie in a lower dimensional skeleton of $\mathcal{B}_x(M''_k)$ 
in the boundary of a fibered face.  The Euler classes used
above are linear functionals dual to each face, and
therefore are useful for distinguishing the faces of the unit
norm ball. 
\end{proof}

\section{Applications}
The following result was observed in \cite[Section 5]{KPV06} to
follow if arithmetic 3-manifolds defined by quadratic
forms are virtually fibered. 

\begin{theorem}
All arithmetic uniform lattices of simplest type ({\it i.e.} defined
by a quadratic form) with
$n\geq 4$ are
not coherent.
\end{theorem}

A group is {\it FD} if the finite representations are dense
in all unitary representations with respect to the Fell
topology. The following result follows
from \cite[Corollary 2.5, Lemma 2.6]{LubotzkyShalom04}.
This generalizes  \cite[Theorem 2.8 (1)]{LubotzkyShalom04},
which proves that $\SL(2,\ZZ[i])$ and $\SL(2,\ZZ[\sqrt{-3}])$
have property FD. The following theorem was pointed out
 by Alan Reid to be a  corollary of Theorem \ref{vquadfiber}. 

\begin{theorem} \label{FD}
Arithmetic Kleinian groups defined by quadratic forms and
reflection groups are FD.
\end{theorem}

\section{Conclusion}
It seems natural to extend Thurston's conjecture to the class of aspherical
3-manifolds which have at least one hyperbolic piece in the geometric decomposition. 

\begin{conjecture} \label{vfiber}
If $M$ is an irreducible orientable compact 3-manifold with $\chi(M)=0$ which is not a graph manifold, then 
there is a finite-sheeted cover $\tilde{M}\to M$ such that $\tilde{M}$ fibers over $S^1$.
\end{conjecture}

{\bf Remark: } There are closed graph manifolds which do not 
virtually fiber over $S^1$. The classic examples are Seifert fibered spaces for which the rational Euler class of the  Seifert fibration is non-zero and
the base orbifold has Euler characteristic non-zero \cite{OrlikRaymond69};
see also \cite[Theorem F]{Neumann97} for examples of irreducible graph
manifolds which do not virtually fiber. 
If an irreducible manifold with a non-trivial geometric decomposition fibers over $S^1$,
then each of the geometric factors of the torus decomposition will also fiber, so
it is certainly necessary for the above conjecture to be true that all of the geometric
pieces in the torus decomposition virtually fiber.
However, the obstructions to virtual fibering vanish when a graph
manifold has boundary or a metric of non-positive curvature \cite{BuyaloSvetlov04}. 
It would be interesting to know if a graph manifold $M$ which virtually fibers has
$\pi_1 M$ virtually RFRS.
 
As mentioned before, the following conjecture is a natural analogue of the virtual
fibering conjecture for taut sutured manifolds. 
\begin{conjecture}
Let $(M,\gamma)$ be a taut sutured manifold such that $\chi(M)<0$. Then there is a finite-sheeted
cover $(\tilde{M},\tilde{\gamma}) \to (M, \gamma)$ such that $(\tilde{M},\tilde{\gamma})$
has a depth one taut oriented foliation. 
\end{conjecture}

The following conjecture is a natural generalization of the virtual fibering
conjecture. Once a manifold $M$ fibers and $\b_1(M)>1$,
then $M$ will fiber over the circle in infinitely many different ways.
But the fibrations coming from a single face of the Thurston
norm are closely related to each other, so one
can ask how many unrelated fiberings are there? Clearly,
the following conjecture is stronger than Conjecture \ref{vfiber}.
\begin{conjecture} \label{fiberedfaces}
If $M$ is irreducible, compact orientable with $\chi(M)=0$ and not a graph manifold, then for
any $k>0$, there is a finite index cover  $M_k\to M$ with $k$
fibered faces of $\mathcal{B}_x(M_k)$. 
\end{conjecture}

{\bf Question:} Which 3-manifolds have RFRS or virtually RFRS fundamental group?
Clearly Seifert fibered spaces with non-zero Euler class of
fibration and non-zero Euler characteristic of base surface
cannot have any finite cover with RFRS fundamental group.
Boileau and Wang have examples derived from this which
cannot have RFRS fundamental group \cite{BoileauWang96}. Sol and Nil manifolds are virtually fibered, but not virtually
RFRS. 
Also, knot complements cannot have RFRS fundamental group,
but it is still possible that they are virtually RFRS. Every
hyperbolic 3-manifold has a finite sheeted cover which is
pro-$p$ \cite{Lubotzky88}, but in general a tower of $p$-covers will
not satisfy the rationality condition of the RFRS definition. 
One strategy to try to prove Conjectures \ref{vfiber} and \ref{fiberedfaces} would
be to show that an irreducible 3-manifold $M$ which is not
a graph manifold has $\pi_1(M)$ virtually a subgroup of a right-angled
Coxeter group, and therefore is virtually RFRS. As noted in \cite{HaglundWise07}, this
is equivalent to showing that $\pi_1(M)$ is LERF and that
$\pi_1(M)$ acts properly and cocompactly on a CAT(0)
cube complex. Having such an action would follow
from the existence of enough immersed $\pi_1$-injective
surfaces to bind every element of $\pi_1 M$. 

It would be interesting to try to compute the index of the
cover of a Bianchi group for which the group will virtually 
fiber. Similarly for all-ideal right-angled reflection groups,
for which we suspect there is a very small finite-index
cover which fibers, possibly a four-fold manifold cover. 

\def\cprime{$'$} \def\cprime{$'$} \def\cprime{$'$}
\providecommand{\bysame}{\leavevmode\hbox to3em{\hrulefill}\thinspace}
\providecommand{\href}[2]{#2}

\end{document}